\documentclass[11pt]{amsart}
\usepackage{amssymb,amsmath,amsthm,amscd,mathrsfs,graphicx}
\usepackage[matrix,arrow,tips,curve]{xy}
\usepackage[english]{babel}
\usepackage[leqno]{amsmath}
\usepackage{amssymb,amsthm}
\usepackage{amscd}
\usepackage{enumerate}
\usepackage{layout}

\newtheorem{theorem}{Theorem}[section]
\newtheorem{convention}[theorem]{}
\newtheorem{lemma}[theorem]{Lemma}
\newtheorem{proposition}[theorem]{Proposition}
\newtheorem{corollary}[theorem]{Corollary}
\newtheorem{question}[theorem]{Question}
\newtheorem{fact}[theorem]{Fact}
\newtheorem{theoremalpha}{Theorem}

\theoremstyle{definition}
\newtheorem{definition}[theorem]{Definition}

\numberwithin{equation}{section}

\newcommand{\calA}{{ \mathcal A}}
\newcommand{\calB}{{ \mathcal B}}

\newcommand{\calE}{{ \mathcal E}}
\newcommand{\calF}{{ \mathcal F}}
\newcommand{\calG}{{ \mathcal G}}

\newcommand{\calI}{{ \mathcal I}}
\newcommand{\calJ}{{ \mathcal J}}
\newcommand{\calK}{{ \mathcal K}}
\newcommand{\calL}{{ \mathcal L}}
\newcommand{\calM}{{ \mathcal M}}
\newcommand{\calN}{{ \mathcal N}}
\newcommand{\calO}{{ \mathcal O}}

\newcommand{\bbN}{{\mathbb N}}

\newcommand{\bbP}{{\mathbb P}}

\newcommand{\Hom}{\operatorname{Hom}}
\newcommand{\ShEnd}{\underline{\operatorname{End}}}
\newcommand{\Aut}{\operatorname{Aut}}

\newcommand{\Ext}{\operatorname{Ext}}
\newcommand{\ShExt}{\underline{\operatorname{Ext}}}
\newcommand{\Tor}{\operatorname{Tor}}
\newcommand{\Spec}{\operatorname{Spec}}

\newcommand{\red}{\operatorname{red}}
\newcommand{\s}{\operatorname{s}}

\makeatletter
\@namedef{subjclassname@2010}{ \textup{2010} Mathematics Subject 
Classification} \makeatother

\begin{document}

\title{Moduli of generalized line bundles on a ribbon}

\author{Dawei Chen}
\address{Department of Mathematics, Statistics and Computer Science, University of Illinois at Chicago, 851 S Morgan Street, 
Chicago, IL 60607}
\email{dwchen@math.uic.edu}

\author{Jesse Leo Kass}
\address{Department of Mathematics, University of Michigan, 530 Church Street, Ann Arbor, MI 48109}
\email{kass@math.harvard.edu}

\thanks{The second author was supported by NSF grant DMS-0502170.}
\date{\today}

\begin{abstract}
A ribbon is a first-order thickening of a non-singular curve.  Motivated by a question of Eisenbud and Green, we show that
a compactification of the moduli space of line bundles on a ribbon is given by the moduli space of semi-stable sheaves. We then describe the geometry of this space, determining the irreducible components, the connected components, and the smooth locus.
\end{abstract}

\date{}

\maketitle

\tableofcontents

\section{Introduction}
This paper describes the moduli space of slope semi-stable sheaves on a ribbon.  A ribbon is a first-order thickening 
of a non-singular curve, and in the context of this paper, their study began with work of Bayer--Eisenbud (\cite{bayer95}) and 
Eisenbud--Green (\cite{green95}), where ribbons were used to study Green's conjecture on linear series using degeneration techniques. More recently, the first named author used these techniques to study Brill--Noether loci (\cite{dawei}).  Given a linear series on a non-singular curve, the approach is to specialize the curve to a ribbon and then specialize the linear series to a series on the ribbon.  Here one encounters a difficulty: it is only possible to specialize the linear series if the specialization is allowed to be a generalized linear series.  Recall that a linear series is a line bundle
$\calI$ together with a subspace of the space of global sections.  Eisenbud and Green  defined a \emph{generalized} linear series on a ribbon by allowing $\calI$ to be a more general coherent sheaf, termed a generalized line bundle.  In \cite{green95}, the authors raised the question: \textbf{does there exist a moduli space of generalized line bundles?}
 
Eisenbud and Green observed that the set of all generalized line bundles of fixed degree cannot be parameterized by an algebraic $k$-scheme because the class of such sheaves is unbounded.  Let $X$ be a ribbon.  Given a finite birational morphism $f \colon X' \to X$ of ribbons, the direct image $\calI := f_{*}(\calI')$ of a line bundle $\calI'$
is a generalized line bundle, and varying over all $f$, we obtain all generalized line bundles.  The genus 
of $X'$ is an important invariant of $\calI$, so following Dr{\'e}zet (\cite[\S5.4]{drezet08}), we define the index $b(\calI)$ of $\calI$
by $b(\calI) := g(X) - g(X')$.  The index of a generalized line bundle can be arbitrarily negative, and unboundedness follows.  
The exact question posed by Eisenbud and Green (\cite[pg.~758, midpage]{green95}) is: \textbf{if $X$
is a rational ribbon (i.e. $X_{\text{red}} \cong \bbP^1$), then is it possible to compactify the moduli space
of degree $0$ line bundles on $X$ by a moduli space of degree $0$ generalized line bundles with non-negative
index?}

Requiring that the index is non-negative is one way to recover boundedness. Another way is to impose the condition of slope semi-stability. In the general theory of moduli of 
sheaves, it has long been recognized that many natural classes of sheaves are  unbounded, but one can recover boundedness by considering sheaves that satisfy slope semi-stability.  In great generality, Simpson (\cite{simpson}) has constructed moduli spaces $\operatorname{M}(\calO_{X}, P)$ of semi-stable sheaves with Hilbert polynomial $P$ on a polarized scheme $(X, \calL)$.  
$\operatorname{M}(\calO_{X}, P)$  is a coarse moduli space in the sense that non-isomorphic sheaves may correspond to the same point of $\operatorname{M}(\calO_{X}, P)$.  Sheaves 
satisfying the stronger condition of stability sweep out an open locus $\operatorname{M}_{\text{s}}(\calO_{X}, P) \subset \operatorname{M}(\calO_{X}, P)$, and this 
locus is the fine moduli space of stable sheaves.  The goal of this paper is to describe the
Simpson moduli spaces $\operatorname{M}(\calO_{X}, P)$ and $\operatorname{M}_{\text{s}}(\calO_{X}, P)$  
when $X$ is a ribbon, with an eye towards addressing the question posed by Eisenbud and Green.

Given a polarized ribbon $(X, \calL)$, we study the Simpson moduli space parameterizing semi-stable sheaves
with Hilbert polynomial $P_{d}(t) := \deg(\calL) t + d + 1 -g$, the Hilbert polynomial of a degree $d$ line bundle.  
Here $g$ is the genus of $X$.  There are two types of sheaves that are parameterized by $\operatorname{M}(\calO_{X}, P_d)$:
generalized line bundles on $X$ and rank $2$ vector bundles on $X_{\red}$.   More precisely, we have
\begin{theoremalpha}\label{Thm: Main I}
	Let $(X, \calL)$ be a polarized ribbon.  Set $g$ equal to the genus of $X$ and $\bar{g}$ equal to the genus of $X_{\red}$.
	If $\calF$ is a coherent sheaf on $X$, then $\calF$ corresponds to a point of
	$\operatorname{M}(\calO_{X}, P_d)$ (resp. $\operatorname{M}_{\s}(\calO_{X}, P_d)$)
	if and only if $\calF$ is isomorphic to one of the following sheaves:
	\begin{itemize}
		\item a degree $d$ generalized line bundle $\calI$ of index less than or equal to  (resp. strictly less than) $1 + g - 2 \bar{g}$;
		\item the direct image $i_* \calE$ of a rank $2$, slope semi-stable (resp. slope stable) vector bundle
			on $X_{\red}$ of degree $d + 2 \bar{g}-1-g$.  Here $i: X_{\red} \hookrightarrow X$ is the inclusion map. 
	\end{itemize}
\end{theoremalpha}
This is a reformation of Theorem~\ref{Thm: StabilityThm}.

Consider the special case where $X$ is a rational ribbon (i.e. $\bar{g}=0$), which was the case considered in \cite{bayer95}. 
Theorem~\ref{Thm: Main I} then asserts that the stable locus $\operatorname{M}_{\s}(\calO_{X}, P_0)$ of the Simpson moduli space is the fine moduli 
space parameterizing generalized line bundles of index at most $g$.  When $g$ is even, every semi-stable sheaf
is stable, and $\operatorname{M}_{\s}(\calO_{X}, P_0)$ is projective.  Otherwise, $\operatorname{M}_{\s}(\calO_{X}, P_0)$
is not projective, and its complement in the projective scheme $\operatorname{M}(\calO_{X}, P_0)$ consists of a single point.  In other words, \textbf{the Simpson moduli space satisfies the properties Eisenbud and Green ask for precisely when $g$ is even}.  This result is restated as Corollary~\ref{Cor: SimpSpaceRatlRibb} in the body of the text.

We prove several results about the geometry of $\operatorname{M}(\calO_{X}, P_d)$.  The following theorem enumerates
the irreducible components of $\operatorname{M}(\calO_{X}, P_d)$.
\begin{theoremalpha}\label{Thm: Main II}
	Let $X$ be a ribbon.  Denote by $g$ the genus of $X$, $\bar{g}$  the genus of $X_{\red}$, and set
	\begin{displaymath}
		n := 
		\begin{cases}
			\lfloor (g+2)/2 \rfloor - \bar{g}  & \text{ if $d$ is even;}\\
			\lfloor (g+1)/2 \rfloor - \bar{g} & \text{ if $d$ is odd.}
		\end{cases}
	\end{displaymath}

	Assume there exists a stable generalized line bundle of degree $d$ (i.e. $g > 2 \bar{g}-1$ holds).  Then  
	$\operatorname{M}( \calO_{X}, P_d)$ has exactly $n$  irreducible components of dimension $g$ whose general element corresponds to a generalized line bundle.
	
	There is at most one additional component.  When it exists, this component is of dimension $4 \bar{g} - 3$ and the general element  corresponds to a stable rank $2$ vector bundle on $X_{\text{red}}$.  This additional component does not exist when $\bar{g}=0,1$ but does exist when the two conditions $\bar{g} \geq 2$ and $4 \bar{g}-3\geq g$ are satisfied.
\end{theoremalpha}
This is Theorem~\ref{Thm: VectorComponent}.  That theorem, together with Theorem~\ref{Thm: LineBundleComp}, provides a more detailed description of the components.  Theorem~\ref{Thm: Main II} says nothing when $g \le 2 \bar{g}-1$, but this case is also discussed in Theorem~\ref{Thm: LineBundleComp},.

We also compute the connected components of $\operatorname{M}(\calO_{X}, P_d)$.  The statement below is a restatement of 
Theorem~\ref{theorem: Conn}.
\begin{theoremalpha}\label{Thm: Main III}
        For a ribbon $X$, the moduli space $\operatorname{M}(\calO_{X}, P_d)$ is connected. 
\end{theoremalpha}

Finally, we determine the smooth locus of $\operatorname{M}(\calO_{X}, P_d)$.
\begin{theoremalpha}\label{Thm: Main IV}
	Let $X$ be a ribbon.  Set $g$ equal to the genus of $X$ and $\bar{g}$ equal to the genus of $X_{\red}$.  If $\bar{g} \ge 2$
and $g \geq 4 \bar{g}-2$, then the smooth locus of $\operatorname{M}(\calO_{X}, P_d)$ is equal to the open subset of  line bundles on $X$.
\end{theoremalpha}
This is Corollary~\ref{Cor: SmoothLocus}, which is a consequence of the computation of the tangent space to $\operatorname{M}_{\s}(\calO_{X}, P_d)$
at a point (Proposition~\ref{Prop: MainTanSpace}). That computation may be of independent interest.

How do these results compare with results in the literature?  In \cite{inaba}, Michi-Aki Inaba studied the moduli space of stable sheaves on a non-reduced scheme and, in particular, proved results about the local structure of the moduli space of slope stable sheaves on a ribbon (\cite[Thm.~2.6]{inaba}).  Closest
to our results is \cite[Rmk~2.7]{inaba}, which contains tangent space computations similar to Lemma~\ref{Lemma:  TanAtVec}.  

Beginning with \cite{drezet06}, Jean-Marc Dr{\'e}zet has written several papers (\cite{drezet06}, \cite{drezet08}, \cite{drezet09}, \cite{drezet11}) studying slope semi-stable sheaves on a multiple curve, the $n$-th order analogue of a ribbon.  Much of this work focuses on  ``quasi-locally free sheaves," a class of sheaves that 
 includes line bundles, but not generalized line bundles.    Most relevant to this paper is \cite{drezet11} (esp. Thm.~5.4.2),
 which the authors became aware of while preparing the current document. In that paper,  Dr{\'e}zet
 provides sufficient conditions for a pure sheaf of dimension $1$ (there called a ``torsion-free sheaf") on a primitive multiple curve  to be (semi-)stable.  Also relevant are Dr{\'e}zet's classification of pure sheaves with Hilbert polynomial $P_d$
 (\cite[\S5.4]{drezet08})  and  infinitesimal computations appearing in \cite{drezet06} (especially
 \cite[Prop.~8.1.3]{drezet06}, which  includes cases of Lemma~\ref{Lemma:  TanAtGLB}).

Finally, after this paper was written, Robert Lazarsfeld informed the authors that some similar
results can be found in \cite{donagi}.  There the authors study the Hitchin and Mukai systems, 
and the moduli space of stable sheaves associated to a ribbon lying on a $K3$-surface 
arises as the nilpotent cone of the Mukai space. The structure of the nilpotent cone is studied in some detail.  In particular, under the additional hypothesis that the nilradical $\calN$ of
$X$ is isomorphic to the anti-canonical bundle $K^{-1}_{X_{\text{red}}}$ of $X_{\text{red}}$
(so $g = 4 \bar{g}-3$), Theorem~\ref{Thm: Main II} is stated as \cite[Thm~3.2]{donagi}  (though
the proof is not given).

This  paper is organized as follows.  In Section~\ref{Sect: CohSh}, we collect some basic definitions and facts about coherent sheaves on ribbons  that are used later.  Theorem~\ref{Thm: Main I} is proven in Section~\ref{Sect: StableSh}, where the result is the culmination of  results on semi-stability.  Finally, the geometry of $\operatorname{M}(\calO_{X}, P_d)$ is studied in Section~\ref{Sect: Simpson}.  The main results proven in that section are Theorems \ref{Thm: Main II}, \ref{Thm: Main III} and \ref{Thm: Main IV}.

{\bf Acknowledgements.} The authors would like to thank Robert Lazarsfeld for informing them of his paper
\cite{donagi} with Donagi and Ein and for providing helpful expository suggestions.  The authors would also 
like to thank Daniel Erman for expository feedback,  Yusuf Mustopa for enlightening conversations about vector bundles on curves, Matt Satriano for helpful discussions on homological algebra.  

\section*{Conventions}\label{Sec: Conv}
\begin{convention}
	$k$ is an \textbf{algebraically closed field}.
\end{convention}

\begin{convention}
     A \textbf{curve} is an irreducible, projective $k$-scheme of dimension $1$.  
\end{convention}

\begin{convention}
	A \textbf{non-singular curve} is a curve that is smooth over $k$.
\end{convention}

\begin{convention}
     A \textbf{ribbon} $X$ is a curve such that the reduced subscheme $X_{\red}$ is a non-singular curve and the
\textbf{nilradical} $\calN$ is locally generated by a non-zero, square-zero element. 
\end{convention}

\begin{convention}
	A ribbon    $X$ is a \textbf{rational ribbon} if $X_{\red}$ is isomorphic to $\bbP^1$.
\end{convention}

\begin{convention}
      The \textbf{degree} $\deg(\calI)$ of a coherent sheaf $\calI$ on $X$ is
$$ \deg (\calI) := \chi(\calI) - \chi(\calO_X). $$
\end{convention}

\begin{convention}
	The \textbf{genus} $g(X)$ of a curve $X$ is 
	$$g := 1 - \chi(\calO_{X}).$$
\end{convention}

\begin{convention}
      $\eta$ is the \textbf{generic point} of $X$.
\end{convention}

\section{Coherent Sheaves on a Ribbon} \label{Sect: CohSh}
Here we collect the facts about coherent sheaves on ribbons that are needed to describe the Simpson moduli space.  In this section, let $X$ be a fixed ribbon with generic point $\eta$.  

Recall that, by definition, $X$ is a curve with the property that the reduced subcurve $X_{\text{red}}$ is non-singular and the nilradical $\calN$ is locally generated by a single square-zero, but non-zero element.  The nilradical is then square-zero, and hence may be considered as a line bundle on 
 $X_{\text{red}}$.  
 
 The degree of this line bundle can easily be computed from the exact sequence
 \begin{equation} \label{Eq: Conormal}
 	\calN \hookrightarrow \calO_{X} \twoheadrightarrow \calO_{X_{\text{red}}}.
 \end{equation}
We have:
 \begin{displaymath}
 	\deg(\calN) = 2 \bar{g} - 1 - g,
 \end{displaymath}
 where $g$ is the genus of $X$ and $\bar{g}$ the genus of $X_{\text{red}}$.  This quantity plays an important role in 
 many of the results of this paper, appearing, for example, as a bound in Theorem~\ref{Thm: Main I}.
 
 Notice that we have defined the genus $g$ by $g := 1 - \chi(\calO_{X})$.  The genus of a reduced curve $X$ always equals $h^{1}(X, \calO_{X})$, but this equality may fail to hold when $X$ is a ribbon.  Indeed, the two numbers are equal precisely when $\calN$ has no global sections, which holds when
$g$ is sufficiently large but not in general.

We now begin our review of coherent sheaves on a ribbon.  First,  some definitions from \cite{green95}.
\begin{definition} \label{Def: CohSh}
	If $\calF$ is a coherent sheaf on $X$, then we write $d(\calF)$ for the dimension of the support of $\calF$.  We say that
$\calF$ is \textbf{pure} if $d(\calF) = d(\calG)$ for all non-zero subsheaves $\calG \subset \calF$.

	If $\calF$ is a coherent sheaf, then we say that a regular function $f \in H^{0}(U,\calO_{X})$ is
a \textbf{non-zero divisor on $\calF$} if the multiplication map $f \cdot \underline{\phantom{s}} \colon \calF|_{U} \to \calF|_{U}$ is injective.  The sheaf $\calF$ is said
to be \textbf{torsion-free} if every non-zero divisor on $\calO_{X}$ is a non-zero divisor on $\calF$.  The sheaf $\calF$ is said to be \textbf{rank $1$} if the generic 
stalk $\calF_{\eta}$ is isomorphic to $\calO_{X, \eta}$. A \textbf{generalized line bundle} is a coherent sheaf that is rank $1$ and torsion-free.
\end{definition}

\begin{lemma} \label{Lemma: PureIsTorsionFree}
	Let $\calI$ be a coherent sheaf satisfying $d(\calI) =1$.  Then $\calI$ is torsion-free if and only if $\calI$ is pure.
\end{lemma}
\begin{proof}
	First, we assume $\calI$ is pure and prove that it is torsion-free.  Given a non-empty open subset $U \subset X$, a non-zero divisor $f \in H^{0}(U, \calO_{X})$
and a local section $s \in H^{0}(U, \calI)$ satisfying $f s =0$, we will show $s=0$.  Consider the annihilator 
$\operatorname{ann}(s)$.  The quotient $\calO_{U}/\operatorname{ann}(s)$ is isomorphic the submodule of $\calI|_{U}$ generated by $s$, so by purity, there are only two possibilities for the quotient: either its support equals $X$ or it is zero.  We can eliminate the first case.  Indeed, as $\operatorname{ann}(s)$ contains a non-zero divisor,
it is not contained in the nilradical, hence is not contained in some maximal ideal.  The corresponding point of 
$X$ does not lie in the support of $\calO_{X}/\operatorname{ann}(s)$, so this quotient must be zero.
Equivalently, $\operatorname{ann}(s)$ is the unit ideal and $s=0$.  This proves that $\calI$ is torsion-free.

 	Now we show that torsion-free implies pure.  Given a non-zero submodule $\calJ \subset \calI$  of a torsion-free module $\calI$, we will show $d(\calJ)=1$.  First, pick an affine open subset $U \subset X$ that meets the support of $\calJ$.  The integers $d(\calJ)$ and $d(\calJ|_{U})$ coincide, and $d(\cdot)$ does not increase if we replace $\calJ|_{U}$ with a 
submodule.  Thus, by replacing $\calJ|_{U}$ with a non-zero cyclic submodule, we can assume $\calJ|_{U}$ is cyclic.  Say $s \in H^{0}(U, \calJ)$ generates $\calJ|_{U}$.

	What are the possibilities for $\operatorname{Supp}(s)$? We claim that $\operatorname{Supp}(s)=U$.  
This support is non-empty because $s$ is non-zero, so it is enough to show that the only irreducible component
of $\operatorname{Supp}(s)$ is $U$.  Suppose $\mathfrak{p}$ is a prime ideal corresponding to such an
irreducible component. Then $\mathfrak{p} = \operatorname{ann}(f s)$
for some $f \in H^{0}(U, \calO_{X})$.  But $\calI$ is torsion-free, so $\mathfrak{p}$ is a union of zero-divisors and 
hence contained in the nilradical.  The only such prime ideal is the nilradical itself, so we may conclude that
$\operatorname{Supp}(s)$ equals $U$.  This proves that $d(\calJ)=1$ and thus that $\calI$ is pure.
\end{proof}

One application of Lemma~\ref{Lemma:  PureIsTorsionFree} is the following result, which provides an alternative 
characterization of generalized line bundles.

\begin{lemma} \label{Lemma:  EndToGLB}
	Let $X$ be a ribbon and $\calI$ a pure sheaf whose generic stalk is an 
$\calO_{X, \eta}$-module of  length $2$.   Then $\calI$ is a generalized line bundle if and only if the natural map
\begin{equation} \label{Eqn: ShEndMap}
	\calO_{X} \to \ShEnd(\calI)
\end{equation}
 is injective.
\end{lemma}
\begin{proof}
	Let $X$ and $\calI$ be given.  We have just shown that
torsion-free is equivalent to pure (Lemma~\ref{Lemma:  PureIsTorsionFree}),
so the content of the lemma is that $\calI$ is rank $1$ if and only if the map \eqref{Eqn: ShEndMap}
is injective.  This map factors as
\begin{displaymath}
	\begin{CD}
		\calO_{X}		@>>>	\ShEnd(\calI)	\\
		@VVV					@VVV			\\
		\calO_{X,\eta}		@>>>	\ShEnd(\calI_{\eta}),
	\end{CD}
\end{displaymath}
where $\eta$ is the generic point. 

When $\calI$ is a generalized line bundle, the vertical maps are injective (because $\calO_{X}$ and $\calI$ are pure)
and $\calO_{X, \eta} \to \ShEnd(\calI_{\eta})$ is an isomorphism (because $\calI$ is 
rank $1$).  We may conclude that \eqref{Eqn: ShEndMap} is injective.

Conversely, assume this map is injective.  Pick a generator $\epsilon$ of the generic 
stalk $\calN_{\eta}$.  By hypothesis, multiplication by $\epsilon$ on $\calI_{\eta}$
is not the zero map, so there exists an element $s_0 \in \calI_{\eta}$ with
$\epsilon s_0 \ne 0$.

We claim that the natural map $\calO_{X, \eta} \to \calI_{\eta}$ given by
$f \mapsto f s_0$ is an isomorphism.  The kernel of this map is 
$\operatorname{ann}(s_0)$, which is contained in $\calN_{\eta}$ (because $\calI$ is pure). 
In fact, this containment is proper because $\epsilon s_0 \ne 0$.  The only ideal 
properly contained in the nilradical is the zero ideal, proving injectivity.
Surjectivity also follows:  both the module $\calI_{\eta}$ and the
submodule generated by $s_0$ have length $2$, so they
must coincide.  We may conclude that $\calI_{\eta}$ is isomorphic to
$\calO_{X, \eta}$, and the poof is complete.
\end{proof}

The previous two results allow us to easily classify the pure sheaves whose generic stalk has length $2$. The classification
result below can be found in the work of Dr{\'e}zet, but we include a proof for the sake of completeness. 

\begin{proposition}[{\cite[\S5.4]{drezet06}}] \label{Prop: PureSh}
	Let $X$ be a ribbon and $\calI$ a pure sheaf whose generic stalk is an $\calO_{X,\eta}$-module of length $2$.  
Then $\calI$ is isomorphic to one of the following:
\begin{enumerate}
	\item a generalized line bundle;
	\item the direct image $i_{*}(\calE)$ of a rank $2$ vector bundle $\calE$ on $X_{\red}$ under the inclusion map $i: X_{\red}\hookrightarrow X$.
\end{enumerate}
\end{proposition}
\begin{proof}
	Given $\calI$, consider the natural map $\calO_{X} \to \ShEnd(\calI)$.
The kernel of this map consists of zero-divisors (Lemma~\ref{Lemma:  PureIsTorsionFree}), hence is contained in the nilradical $\calN$.  
There are only two ideals with this property: the 
zero ideal and the nilradical itself.  When the kernel is the zero ideal, 
Lemma~\ref{Lemma:  EndToGLB} states that $\calI$ is a generalized line bundle.  
Thus, we focus on the case where the kernel is $\calN$.

In this case, we can consider $\calI$ as a module $\calE$ over 
$\calO_{X}/\calN = \calO_{X_{\text{red}}}$.  Certainly, $\calE$ then
satisfies $\calI = i_{*}(\calE)$.  The module $\calE$ is pure, hence locally free
(by, say, the Auslander--Buchsbaum formula) and the
generic rank is $2$ by hypothesis.  In other words, $\calE$ is a rank $2$ vector bundle, completing
the proof.
\end{proof}

We conclude this section with a discussion of the sheaves appearing in Proposition~\ref{Prop: PureSh}.
At this level of generality, there is not much to be said about the classification of rank $2$ vector bundles 
on $X_{\text{red}}$.  When $X_{\text{red}}$ has genus $0$, every vector bundle is 
a direct sum of line bundles. But in general, there are non-split rank $2$ vector bundles and the classification is more complicated.  In Section~\ref{Sect: StableSh}, 
we introduce stability conditions and the semi-stable vector bundles on $X_{\text{red}}$ are coarsely parameterized by the corresponding Simpson moduli space.

There is a classification of generalized line bundles on $X$, which can be found in \cite[Thm 1.1]{green95}.  Given a Cartier divisor $D$ on $X_{\text{red}}$, we can consider $D$ as a closed subcheme of $X$ and form the blow-up $f \colon X' := \operatorname{Bl}_{D}(X) \to X$. The blow-up $X'$ is itself a ribbon and $f$ is a finite morphism.
In fact, the quotient $f_{*}(\calO_{X'})/\calO_{X}$ is isomorphic to $\calO_{D}$.  If $\calL'$ is a line bundle on $X'$, then the direct image $\calI := f_{*}(\calL')$ is a generalized line bundle.  Theorem~1.1 of \cite{green95}
states that every generalized line bundle is of this form for a unique blow-up $f \colon X' \to X$ and a unique line bundle $\calL'$.  Given
$\calI$, we call $f \colon X' \to X$ the \textbf{associated blow-up}.  Following Dr{\'e}zet (\cite[\S5.4]{drezet08}, we use the Green-Eisenbud result to define an invariant of $\calI$.  

\begin{definition}\label{Def: Local}
	Let $\calI$ be a generalized line bundle on $X$.  Say $\calI := f_{*}(\calI')$ for the blow-up $f \colon X' = \operatorname{Bl}_{D}(X) \to X$
and the line bundle $\calI'$ on $X'$.  As a divisor on $X_{\text{red}}$, write $D = \sum n_{p} p$.  Given a point $p_0 \in X$, we define the 
\textbf{local index} of $\calI$ at $p_0$, written $b_{p_0}(\calI)$, to be $b_{p_0}(\calI) := n_{p_0}$.  The \textbf{local index sequence} is the collection
$\{ b_{p_0}(\calI) \colon b_{p_0}(\calI) \ne 0 \}$, and the \textbf{index} of $\calI$ is defined by $b(\calI) := \sum b_{p}(\calI)$.
\end{definition}
The integers $b(\calI)$ and $b_{p_0}(\calI)$ can be defined more intrinsically.  The index $b(\calI)$ is the length of $\ShEnd(\calI)/\calO_{X}$,
while $b_{p_0}(\calI)$ is the length of the localization $\ShEnd(\calI_{p_0})/\calO_{X,p_0}$.  The index of a degree $d$ generalized line
bundle cannot be an arbitrary integer.  

\begin{fact} \label{Fact: TypeRestrict}
	If $\calI$ is a generalized line bundle, then $\deg(\calI) - b(\calI)$ is even.
\end{fact}
\begin{proof}
	If $\calI$ is a line bundle, by definition $b(\calI) = 0$. Tensoring the sequence $\calN \hookrightarrow \calO_{X} \twoheadrightarrow \calO_{X_{\text{red}}}$
with $\calI$ and taking Euler characteristics, we see that 
$$\deg(\calI) = 2 \cdot \deg(\calI \otimes \calO_{X_{\text{red}}}), $$ 
which is an even number. For the general case, write $\calI = f_{*}(\calI')$ for some line bundle $\calI'$ on a blow-up $f \colon X' \to X$.  The Euler characteristics $\chi(\calI)$ and
$\chi(\calI')$ are equal. Writing out these numbers, we see 
\begin{align*}
	\deg(\calI) &= \deg(\calI') + g - \bar{g} \\
			&= \deg(\calI') + b(\calI).
\end{align*}
	In particular, $\deg(\calI)-b(\calI)$ is even.
\end{proof}

One consequence of the classification theorem of Green--Eisenbud is that the generalized Jacobian, or moduli space of 
degree $0$ line bundles,  acts transitively on the set of generalized
line bundles with fixed local index sequence.  This fact will be used later, so we record it.
\begin{lemma} \label{Lemma: ActionIsTrans}
	Let $X$ be a ribbon.  If $\calI_1$ and $\calI_2$ are two generalized line bundles that have the same local index at $p$ for all $p \in X$, then there exists a line bundle
$\calL$ on $X$ such that $\calI_1$ is isomorphic to $\calL \otimes \calI_2$.  Furthermore, any two line bundles with this property differ by 
an element of $\ker(f^{*} \colon \operatorname{Pic}(X) \to \operatorname{Pic}(X'))$, where $f \colon X' \to X$ is the blow-up associated
to $\calI_i$.
\end{lemma} 
\begin{proof}
	If $f \colon X' \to X$ is the blow-up as in Definition~\ref{Def: Local}, then we can write $\calI_i = f_{*}(\calL_i)$ for 
line bundles $\calL_1$ and $\calL_2$ on $X'$.  The map $f^{*} \colon \operatorname{Pic}(X) \to \operatorname{Pic}(X')$ is surjective,
so we can find a line bundle $\calL$ such that $f^{*}(\calL) = \calL_1 \otimes \calL_2^{-1}$.  An application of the projection formula
shows that this line bundle satisfies the desired conditions, and that any other line bundle $\calM$ with this property must
satisfy $f^{*}(\calM) \cong f^{*}(\calL)$.  
\end{proof}
The structure of the generalized Jacobian $J^{0}(X)$ of $X$ can be described explicitly: the pullback map 
$i^{*} \colon J^{0}(X) \to J^{0}(X_{\text{red}})$ from the generalized Jacobian of $X$ to the Jacobian of $X_{\text{red}}$
is surjective with kernel equal to the vector space group associated to $H^{1}(X_{\text{red}}, \calN)$.  

This description of
$J^{0}(X)$ can  be found in \cite[\S~2.3]{drezet08} or \cite[Sect.~9.2]{bosch}.  One approach is to use the description of $J^{0}(X)$ as the  identity  component of the moduli space $\operatorname{Pic}(X)$ of line bundles on $X$.  In
\cite{bosch},  $\operatorname{Pic}(X)$ is described as the scheme that represents the \'{e}tale sheaf $R^{1}g_{*}(\calO_{X}^{\ast})$, where $g \colon X \to \Spec(k)$ is the structure morphism, and then  computed using long exact sequence associated to the ``exponential sequence"
\begin{displaymath}
	\calN \hookrightarrow \calO_{X}^{\ast} \twoheadrightarrow \calO_{X_{\text{red}}}^{\ast}.
\end{displaymath}
In any case, we see that the dimension of $J^{0}(X)$ equals $h^{1}(X, \calO_{X})$ and $J^{0}(X)$ is 
not proper once $g \ge \bar{g}+1$ (as then $H^{1}(X_{\text{red}}, \calN) \ne 0$).  These two facts
will be used several times in Section~\ref{Sect: Simpson}.

If we pass from $X$ to its completed local ring at a point, then we can describe generalized line bundles more explicitly.
\begin{definition} \label{Def: LocalType}
	Set $\calO_{0} := k[[s, \epsilon]]/(\epsilon^2)$.  Define the \textbf{$n$-th blow-up algebra} to be
	$\calO_{n} :=  k[[ \tilde{s}, \tilde{\epsilon}]]/(\tilde{\epsilon}^2)$, considered as an  
	$\calO_0$-algebra via $s \mapsto \tilde{s}$, $\epsilon \mapsto \tilde{\epsilon} \tilde{s}^n$.
	We write $I_n$ for $\calO_n$, considered as an $\calO_0$-module.
\end{definition}
One may easily compute that $\calO_0 \to \calO_{n}$ is the algebra extension corresponding to the blow-up of 
$\calO_0$ along $(s^n, \epsilon)$.  Given a blow-up $f \colon X' := \operatorname{Bl}_{D}(X) \to X$ as in Definition~\ref{Def: Local}
and a point $q \in X'$ mapping to $p \in X'$, the induced map $\widehat{\calO}_{X, p} \to \widehat{\calO}_{X', q}$
can be identified with $\calO_{0} \to \calO_{n_{p}}$.   Combining this observation with \cite[Thm.~1.1]{green95},
 we can deduce the following lemma.
\begin{lemma}
	Let $X$ be a ribbon, $p_0 \in X$ a point and $\calI$ a generalized line bundle.  Fix an isomorphism between $\calO_{0}$ and
$\widehat{\calO}_{X, p_0}$.  Then under this isomorphism, $\calI \otimes \widehat{\calO}_{X, p_0}$ is identified with 
a module isomorphic to $I_n$, where $n = b_{p_0}(\calI)$.
\end{lemma}
\begin{proof}
	Write $\calI = f_{*}(\calL')$ for the blow-up $f \colon X' := \operatorname{Bl}_{D}(X) \to X$ and a line bundle $\calL'$ on $X'$.
If $q \in X'$ is the unique point mapping to $p$, then the identification of $\widehat{\calO}_{X,p}$ with $\calO_0$ extends to 
an identification of the $\hat{f} \colon \widehat{\calO}_{X,p} \to \widehat{\calO}_{X', q}$ map on completed local rings with 
$\calO_0 \to \calO_n$.  In particular, $\calI \otimes \widehat{\calO}_{X,p}$ is identified with the direct image of a line bundle
on $\calO_n$, and such a module is isomorphic to $\calO_n$ itself.
\end{proof}

For later computations, it is convenient to have an alternative description of $I_n$.  This module can also be described as the
ideal $(s^n, \epsilon)$ of $\calO_0$.  One isomorphism from $I_n$ to $(\epsilon, s^n)$ is given by the
map sending $1 \in I_n$ to $s^n$ and $\tilde{\epsilon} \in I_n$ to $\epsilon$.  This common module admits the following 
presentation:
\begin{displaymath}
	\langle e, f \colon \epsilon f=0, s^n f = \epsilon e \rangle.
\end{displaymath}
Here the element $e$ corresponds to $1 \in I_n$ and $f$ corresponds to $\tilde{\epsilon}$.  This presentation in fact extends to the 
period resolution: 
			\begin{equation} \label{Eqn: FreeRes}
				\cdots \longrightarrow \calO_{0}^{2} \stackrel{ \begin{pmatrix} \epsilon & s^{n} \\ 0 & -\epsilon \end{pmatrix} }{\longrightarrow} \calO_{0}^{2} \stackrel{ \begin{pmatrix} \epsilon & s^{n} \\ 0 & -\epsilon \end{pmatrix} }{\longrightarrow} \calO_{0}^2 \longrightarrow I_{n} \longrightarrow 0. 
			\end{equation}
Later, we shall use this presentation to describe how $I_n$ deforms.

\section{Stable Sheaves on a Ribbon} \label{Sect: StableSh}
Here we study the stability of generalized line bundles.  In general, the stability of a coherent sheaf on a projective scheme
is a condition defined in terms of an auxiliary ample line bundle $\calL$.  However, on a ribbon the stability condition
is independent of $\calL$, as we will see. To fix ideas, let us first work with a polarized ribbon $(X, \calL)$.  As before, we write
$g$ for the genus of $X$, $\bar{g}$ for the genus of $X_{\text{red}}$ and $\calN$ for the nilradical of $X$.
We now recall the definition of the Hilbert polynomial.

Given a coherent sheaf $\calI$ on $X$, the \textbf{Hilbert Polynomial} $ P(\calI,t)$ of $\calI$ with respect 
to $\calL$ is the unique polynomial satisfying 
\begin{displaymath}
	 P(\calI, n) = \chi(\calI \otimes \calL^{\otimes n})  \text{ for all $n \in \bbN$.}
\end{displaymath}
The leading term of $P(\calI, t)$ is particularly significant.  Write 
\begin{displaymath}
	P(\calI,t) = a_0 t^{d}/d! + a_1 t^{d-1}/(d-1)! + \dots + a_d.
\end{displaymath}
 We have
\begin{gather*}
	d = d(\calI), \text{ the dimension of $\operatorname{Supp}(\calI)$;}\\
	a_0 = \deg(\calL) \operatorname{len}(\calI_{\eta}), 
\end{gather*}
where $\operatorname{len}(\calI_{\eta})$ is the length of $\calI_{\eta}$ as an $\calO_{X,\eta}$-module.

If $\calI$ is a sheaf on $X$ with $d(\calI)=1$, then the \textbf{slope} $\mu(\calI)$ of $\calI$ with respect to $\calL$ is defined to be
\begin{align*}
	\mu (\calI) := a_1/a_0.
\end{align*}
We say that $\calI$ is \textbf{slope semi-stable} with respect to $\calL$ if $\calI$ is pure and for all non-zero pure subsheaves $\calJ \subset \calI$ we have
\begin{displaymath}
	\mu(\calJ) \le \mu(\calI).
\end{displaymath}
If this inequality is always strict, we say $\calI$ is \textbf{slope stable}. A semi-stable sheaf that is not stable is said to be \textbf{strictly semi-stable}.  An equivalent formulation of semi-stability is that $\mu(\calI) \le \mu(\calJ)$ for 
all quotients $\calI \twoheadrightarrow \calJ$ with $\calJ$ pure of dimension $d(\calJ) = d(\calI)$, and similarly with stability.  

We have just given the general
definition of semi-stability, but observe that, on a ribbon, the condition is independent of $\calL$.  Indeed, the slope $\mu(\calI)$ depends on $\calL$,
but replacing $\calL$ with a different ample line bundle $\calM$ modifies the slope by a factor of $\deg(\calL)/\deg(\calM)$, so the slope semi-stability 
condition is unchanged.

Given a semi-stable sheaf $\calI$, there exists a filtration $0 = \calI_0 \subset \cdots \subset \calI_n = \calI$ with the property that the successive quotients
$\calI_k/ \calI_{k-1}$ are slope stable of dimension $d(\calI)$ and the slopes $\mu(\calI_{k}/\calI_{k-1})$ are all equal.  This filtration is not unique, but
the associated direct sum
\begin{displaymath}
	\operatorname{Gr}(\calI) := \bigoplus_{k} \calI_{k}/\calI_{k-1}
\end{displaymath}
is unique up to isomorphism. Two coherent sheaves $\calI$ and $\calI'$ are said to be \textbf{Gr-equivalent} if there is an isomorphism $\operatorname{Gr}(\calI) \cong \operatorname{Gr}(\calI')$.  Given a semi-stable sheaf $\calI$, observe
that $\operatorname{Gr}(\calI)$ is also a semi-stable sheaf whose associated graded is $\operatorname{Gr}(\calI)$.

We now specialize to the case of semi-stable sheaves on $X$ with Hilbert polynomial
\begin{equation}
	P_d(t) := \deg(\calL) t + d + 1 -g.
\end{equation}
This is the Hilbert polynomial of a degree $d$ generalized line bundle.  We will show that the stability condition on a generalized
line bundle is controlled by a distinguished quotient.  Following \cite{green95},  we make the following definition.

\begin{definition}
	The sheaf $\bar{\calI}$ associated to a generalized line bundle $\calI$ is defined to be the \textbf{maximal torsion-free quotient} of $\calI \otimes \calO_{X_{\text{red}}}$.
\end{definition}

The sheaf $\bar{\calI}$ is a line bundle on $X_{\text{red}}$ of degree $(\deg(\calI) - b(\calI) )/2$.  By the projection formula, the Hilbert polynomial of
$i_{*}(\bar{\calI})$ with respect to $\calL$ is equal to the Hilbert polynomial of $\bar{\calI}$ with respect to $i^{*}(\calL)$.  This common polynomial is
\begin{equation}
	\frac{\deg(\calL)}{2} t + (\deg(\calI)-b(\calI))/2 +1-\bar{g}.
\end{equation}
Using $\bar{\calI}$ as a test quotient sheaf, we conclude that if $\calI$ is semi-stable, then its index must satisfy
\begin{equation} \label{Eqn: StabilityIneq}
	b(\calI) \le 1 + g - 2 \bar{g}.
\end{equation}

In fact, this inequality is sufficient to characterize slope semi-stability.
\begin{lemma}
	Let $(X,\calL)$ be a polarized ribbon and $\calI$ a generalized line bundle of degree $d$.  Then $\calI$ is slope semi-stable with respect to $\calL$ if and
only if Inequality~\eqref{Eqn: StabilityIneq} holds.  Similarly, $\calI$ is slope stable if and only if Inequality~\eqref{Eqn: StabilityIneq} is strict. 
\end{lemma}
\begin{proof}
	Inequality~\eqref{Eqn: StabilityIneq} is equivalent to the slope inequality $\mu(\calI) \le \mu( \bar{\calI} )$, so one implication is clear.  For the converse,
we assume $\mu(\calI) \le \mu(\bar{\calI})$ and then prove $\mu(\calI) \le \mu(\calJ)$ for all pure quotients $q \colon \calI \twoheadrightarrow \calJ$ with $d(\calJ)=1$.  There are
two separate cases to consider: the case where the leading term of $P(\calJ, t)$ is $\deg(\calL)\cdot t$ and the case where it is $\deg(\calL)/2\cdot t$.

	First, suppose the leading term of $P(\calJ, t)$ is $\deg(\calL) \cdot t$.  We claim that $q$ is in fact an isomorphism, and thus there is no slope inequality to check.  We begin by showing that 
$\calJ$ is a generalized line bundle.  The condition on the Hilbert polynomial is equivalent to the condition 
that the generic stalk $\calJ_{\eta}$ has length $2$, so by Proposition~\ref{Prop: PureSh} it is enough to show that $\calJ$ is not the direct image of a
rank $2$ vector bundle on $X_{\text{red}}$.  This, however, is clear: $\calJ_{\eta}$ is generated by a single element (the image of a
local generator of $\calI$), but the direct image of a rank $2$ vector bundle does not have this property.  Having shown that $\calJ$ is a generalized
line bundle, the result follows easily.  Consider the kernel of $q$.  Because both $\calI$ and $\calJ$ are generalized line bundles, $q$ is generically an isomorphism, hence $\ker(q)$ is generically zero.  But $\ker(q)$ is a subsheaf of the pure sheaf $\calI$, so this forces $\ker(q)=0$.  We can conclude that
$q$ is an isomorphism.

	Next we consider the case where the leading term of $P(\calJ, t)$ is $\deg(\calL)/2 \cdot t$.  
To begin, we claim that the kernel $\calK$ of $\calO_{X} \to \ShEnd(\calJ)$ equals $\calN$.  The kernel is certainly contained in $\calN$ because $\calJ$ is pure.  
Furthermore, the only proper $\calO_X$-submodule of $\calN$ is the zero ideal, so it is enough to show that $\calK \ne 0$.  Consider the generic stalk $\calJ_{\eta}$. The assumption on the 
Hilbert polynomial implies that the length of $\calJ_{\eta}$, and hence of every non-zero cyclic submodule, is $1$.  In particular, if $s_0 \in \calJ_{\eta}$
is non-zero, then $\calO_{X, \eta} \cdot s_0 = \calJ_{\eta}/\operatorname{ann}(s_0)$ has length $1$.  Since $\calI$ has length $2$, we may conclude that $\calK_{\eta} = \operatorname{ann}(s_0)$
is non-zero.  This establishes the claim $\calK = \calN$.

Because $\calK = \calN$, we may factor $q$ as $\calI \to i_{*}(\bar{\calI}) \stackrel{\bar{q}}{\rightarrow} \calJ$.
Furthermore, consider $\calJ$ as a module over $\calO_{X}/\calN = \calO_{X_{\text{red}}}$.  If we write $\calF$ for this module, then
$\calF$ satisfies $\calJ = i_{*}(\calF)$.  This module is also a line bundle on $X_{\text{red}}$ because it is locally free (as it is pure) of
generic rank $1$.  Thus, $\bar{q}$ is a surjection between line bundles on $X_{\text{red}}$, so it must be an isomorphism.  In particular, the inequality $\mu(\calI) \le \mu(\calJ)$ is exactly the hypothesis of the lemma.  This completes the proof.
\end{proof}
The lemma also shows that the strictly semi-stable
generalized line bundles are exactly the generalized line bundles of index $b(\calI) = 1 + g - 2 \bar{g}$.  It is natural
to ask what their associated graded modules are.

\begin{lemma} \label{Lemma: FilterGLB}
	Let $(X, \calL)$ be a polarized ribbon and $\calI$ a generalized line bundle of index $b(\calI) = 1 + g - 2 \bar{g}$.  Set $\operatorname{F}_1(\calI)$
to be the kernel of $\calI \to \bar{\calI}$.  Then $0 \subset \operatorname{F}_1(\calI) \subset \calI$ is filtration whose successive quotients are stable
with the same Hilbert polynomial, and 
\begin{displaymath}
	\operatorname{Gr}(\calI) = \operatorname{F}_{1}(\calI) \oplus \bar{\calI}.
\end{displaymath}
\end{lemma}
\begin{proof}
	Both $\operatorname{F}_{1}(\calI)$ and $\bar{\calI}$ are pure of dimension $1$ (in fact, direct images of the line bundles on $X_{\text{red}}$), so it is enough to show
$\mu(\bar{\calI}) = \mu(\operatorname{F}_{1}(\calI))$.  This follows from the work we have already shown.  One computes 
\begin{gather*}
	\mu(\bar{\calI}) = (\deg(\calI) - b(\calI) + 2 - 2 \bar{g} )/\deg(\calL), \\
	\mu(\operatorname{F}_{1}(\calI)) = ( \deg(\calI) + b(\calI) + 2 \bar{g} - 2 g )/\deg(\calL). 
\end{gather*}
It is easy to check that these two numbers are equal precisely when $b(\calI) = 1 + g - 2 \bar{g}$.
\end{proof}
The submodule $\operatorname{F}_{1}(\calI)$ can be described more explicitly.  Say the blow-up 
$f \colon X' \to X$ associated to $\calI$ is given by blowing up the divisor $D \subset X_{\text{red}}$.
Then $\operatorname{F}_{1}(\calI) = \calN \otimes \bar{\calI}(D)$.  (This is \cite[pg.~759, bottom of page]{green95}).

$\operatorname{F}_{1}(\calI) \oplus \bar{\calI}$ is, of course, the direct image of a split rank $2$ vector bundle on $X_{\text{red}}$.    For the sake of completeness, we
should also discuss the slope stability condition on the direct image $i_{*}(\calE)$ of a vector bundle, but there is little to say.  Slope stability
of $i_{*}(\calE)$ is equivalent to slope stability of $\calE$.  This will be used in a later section, so let us record it as a fact.
\begin{fact}
	Let $X$ be a ribbon.  Then the direct image $i_* \calE$ of a rank $2$ vector bundle $\calE$  on $X_{\red}$ is semi-stable if and only
if for every subbundle $\calF \subset \calE$, we have
\begin{equation} \label{Eqn: StableVectorBundle}
	\deg(\calF)/\operatorname{rank}(\calF) \le \deg(\calE)/\operatorname{rank}(\calE).
\end{equation}
In particular, if $X$ is a rational ribbon (i.e. $X_{\red} \cong \bbP^1$), then there are no stable sheaves of this form
and the only semi-stable sheaves are $i_{*}(\calO(e) \oplus \calO(e))$.
\end{fact}

We summarize the results of this section with the following definition and theorem.
\begin{definition}
	Let $X$ be a ribbon.  A coherent sheaf $\calI$ is said to be  \textbf{length $2$ and  semi-stable} if it is one of the following sheaves: 
	\begin{enumerate}
		\item a genefralized line bundle whose index satisfies Inequality~\eqref{Eqn: StabilityIneq}; 
		\item the direct image $i_{*}(\calE)$ of a semi-stable rank $2$ vector bundle $\calE$ on $X_{\text{red}}$.
	\end{enumerate}
	$\calI$ is said to be length $2$ and \textbf{stable} if it is one of the following sheaves: 
		\begin{enumerate}
			\item a generalized line bundle such that \eqref{Eqn: StabilityIneq} is a strict inequality;
			\item the direct image $i_{*}(\calE)$ of a stable rank $2$ vector bundle $\calE$ on $X_{\text{red}}$.
		\end{enumerate}
\end{definition}

\begin{theorem} \label{Thm: StabilityThm}
	Let $(X, \calL)$ be a polarized ribbon.  If $\calI$ is a pure sheaf with Hilbert polynomial $P_{d}(t) = \deg(\calL) t + d + 1 -g$, then $\calI$ is slope semi-stable (resp. stable)
	with respect to $\calL$ if and only if it is a length $2$, semi-stable (resp. stable) sheaf.
\end{theorem}
The special case of rational ribbons is particularly nice.
\begin{corollary} \label{Cor: RatlStability}
	Let $(X, \calL)$ be a polarized rational ribbon.  If $\calI$ is a slope semi-stable sheaf with Hilbert polynomial $P_{d}(t)$, then it is one of the following sheaves: 
	\begin{enumerate}	
		\item  a degree $d$ generalized line bundle  whose index satisfies $b(\calI) \le 1+g$; 
		\item the direct image $i_{*}(\calE)$, where $\calE = \calO( (d-1-g)/2 ) \oplus \calO( (d-1-g)/2 )$. 
	\end{enumerate} 
	The second case can only occur when $d-1-g$ is even.

	If $\calI$ is strictly semi-stable, then it is one of the following sheaves: 
	\begin{enumerate}
		\item a degree $d$ generalized line bundle whose index satisfies $b(\calI)=g+1$;
		\item the direct image $i_{*}(\calE)$, where $\calE = \calO( (d-1-g)/2 ) \oplus \calO( (d-1-g)/2 )$.
	\end{enumerate}
	In particular, any two strictly semi-stable sheaves are $\operatorname{Gr}$-equivalent, and if
	$d-1-g$ is odd, then slope semi-stability is equivalent to slope stability.
\end{corollary}

\section{Moduli of Sheaves on a Ribbon} \label{Sect: Simpson}
Semi-stable generalized line bundles were described in Section~\ref{Sect: StableSh}. The significance of these sheaves is that
they are coarsely parameterized by a moduli space.  Simpson \cite{simpson} has constructed the coarse moduli
space of semi-stable sheaves on an arbitrary polarized projective $k$-scheme, in particular on a ribbon.  In his paper, Simpson works over the complex numbers, but his work has since been generalized to positive characteristic (\cite[Thm.~0.6]{maruyama} or \cite[Thm.~0.2]{langer}).  Here we describe the geometry of the Simpson moduli space.  As before, $X$ will be a ribbon of genus $g$ with nilradical $\calN$.  We write $\bar{g}$ for the genus of $X_{\text{red}}$.  

On a ribbon, we have seen that the slope stability condition is independent of the choice of polarization. But to fix ideas, let us temporarily work with a polarized ribbon $(X, \calL)$.  Given a  polynomial $P(t)$, 
the \textbf{Simpson moduli functor} $\operatorname{M}^{\sharp}(\calO_{X}, P) \colon \text{$k$-Sch} \to \text{Sets}$
is defined by letting the set of $T$-valued points equal the set of isomorphism classes of $\calO_{T}$-flat, finitely presented $\calO_{X_T}$-modules that are fiber-wise
slope semi-stable with Hilbert polynomial $P(t)$.  Inside of the Simpson moduli functor, we can consider the subfunctor $\operatorname{M}^{\sharp}_{\text{s}}(\calO_{X}, P)$
that parameterizes slope stable sheaves.  

The basic existence theorem \cite[Thm.~1.21]{simpson} states that there exists a pair $(\operatorname{M}(\calO_{X}, P), p)$ consisting of a projective scheme $\operatorname{M}(\calO_{X}, P)$ and a natural transformation $p \colon \operatorname{M}^{\sharp}(\calO_{X}, P) \to \operatorname{M}(\calO_{X}, P)$ that universally co-represents
$\operatorname{M}^{\sharp}(\calO_{X}, P)$.  In other words, $p$ is universal with respect to natural transformations from $\operatorname{M}^{\sharp}(\calO_{X}, P)$
to a  $k$-scheme and this property persists under base change by an arbitrary morphism $T \to \operatorname{M}(\calO_{X}, P)$.  Furthermore, $p$ induces a bijection between the $k$-valued 
points of $\operatorname{M}(\calO_{X}, P)$ and the set of $\operatorname{Gr}$-equivalence classes of semi-stable sheaves on $X$ with Hilbert polynomial $P(t)$.  We call $\operatorname{M}(\calO_{X}, P)$ the \textbf{Simpson moduli space}.

We can obtain stronger results by restricting to stable sheaves.  There is an open subscheme $\operatorname{M}_{\text{s}}(\calO_{X}, P)$ of $\operatorname{M}(\calO_{X}, P)$
 whose pre-image under $p$ is $\operatorname{M}^{\sharp}_{\text{s}}(\calO_{X}, P)$. The restriction $\operatorname{M}^{\sharp}_{\text{s}}(\calO_{X}, P) \to 
\operatorname{M}_{\text{s}}(\calO_{X}, P)$ realizes $\operatorname{M}_{\text{s}}(\calO_{X}, P)$ as the scheme that represents the \'{e}tale
sheaf associated to $\operatorname{M}_{\text{s}}(\calO_{X}, P)$. (This is a restatement of \cite[Thm.~1.21(4)]{simpson}.)  The scheme $\operatorname{M}_{\text{s}}(\calO_{X}, P)$
is called the \textbf{stable Simpson moduli space}.

If we specialize to the case where $P(t)$ equals $P_d(t) := \deg(\calL) t + d + 1 - g$, the Simpson moduli space can be described using the results from Section~\ref{Sect: StableSh}: it is the coarse moduli space of length $2$, semi-stable sheaves of degree $d$. The stable locus is the fine moduli space parameterizing those sheaves that are stable.  
In \cite{green95}, Eisenbud and Green asked if it is possible to compactify the Jacobian of a rational ribbon by a moduli space parameterizing generalized line 
bundles of index at most $g$. The Simpson moduli space is a natural candidate for such a compactification;  perhaps surprisingly it has the desired properties precisely when $g$ is even.
\begin{corollary}[Reformulation of Corollary~\ref{Cor: RatlStability}] \label{Cor: SimpSpaceRatlRibb}
	Let $X$ be a rational ribbon.  Then the stable Simpson moduli space $\operatorname{M}_{\s}(\calO_{X}, P_0)$ parameterizes generalized
	line bundles of degree $0$ and index at most $g$.  If $g$ is even, then $\operatorname{M}_{\s}(\calO_{X}, P_0)$ is projective.
	Otherwise, the complement of the stable locus in $\operatorname{M}(\calO_{X}, P_0)$ consists of a single point that
	represents the $\operatorname{Gr}$-equivalence class of the polystable sheaf
	\begin{displaymath}
		i_{*}( \calO( (-1-g)/2) \oplus (-1-g)/2)).
	\end{displaymath}
	This equivalence class consists of the above sheaf and every generalized line bundle of degree $0$ and index $g+1$.
\end{corollary}

We now turn our attention to describing the global and local geometry of the Simpson moduli space.

\subsection{Global Structure}
Having defined $\operatorname{M}(\calO_{X}, P_d)$, we study the global geometry of this space.  
First, we show that it is impossible for a rank $2$ vector bundle on $X_{\text{red}}$ to specialize to a generalized line bundle on $X$.

\begin{lemma} \label{Lemma:  GLBIsOpen}
	Suppose $T$ is a $k$-scheme and $\calI$ a sheaf that represents a $T$-valued point of $\operatorname{M}^{\sharp}(\calO_{X}, P)$.  Define $T_0 \subset T$ 
to be the locus of points $t \in T$ with the property that the restriction of $\calI$ to the fiber $X_{t} := X_{T} \times_{T} \Spec(k(t))$ is a generalized line bundle.  Then
$T_0 \subset T$ is open.
\end{lemma}
\begin{proof}
	The property of being a generalized line bundle is equivalent to the property that $\calN$ acts non-trivially.  A standard 
argument  \cite[12.0.2]{ega43} lets us reduce to showing that the property in question is stable under generalization, which
can be seen directly.  We now make this sketch precise.  

	Fix a non-empty, open affine subset $U \subset X$ with the property that the restriction of the nilradical is 
generated by a single element $\epsilon$ and set $\calI_U := \calI|_{U \times T}$.   Using the classification of length $2$ 
semi-stable sheaves (Proposition~\ref{Prop: PureSh}), we may assert that $T_0$ is the locus of points $t$ with the 
property that the restriction of $\epsilon \cdot \underline{\phantom{s}} \colon \calI_U \to \calI_U$ to $X_t$ is non-zero. 

A finite presentation argument (\cite[8.9.1, 11.2.6]{ega43}) allows us to reduce to the 
case that $T$ is an affine, noetherian scheme.  Having made this reduction, we can cite \cite[Cor.~9.4.6]{ega43}
to assert that $T_0$ is constructible.  To complete the proof, it is enough to show that the complement of 
$T_0$ is stable under specialization. 

Thus, let us assume that $T = \Spec(A)$ is the spectrum of a discrete valuation ring with uniformizer $\pi$ and
that $\epsilon$ acts trivially on the generic fiber of $\calI_U$.  By flatness, we may deduce that $\epsilon$ acts
trivially on $\calI_U$ (because $\calI_U$ injects into its generic fiber) and this property persists upon passing to the 
special fiber. In other words, the complement of $T_0$ is closed under specialization, hence
the proof is complete.
\end{proof}

Next, we compute the dimension of various loci of generalized line bundles in $\operatorname{M}(\calO_{X}, P)$.
\begin{definition} \label{Def: FixedTypeLocus}
	Let $X$ be a ribbon of genus $g$ and $\underline{b} := (b_1, \dots, b_k)$ a (possibly empty) sequence of positive integers satisfying $b_1 + \cdots + b_k \le g - 2 \bar{g}$.
	Define $Z_{\underline{b}} \subset \operatorname{M}_{\text{s}}(\calO_{X}, P_d)$ to be the subset of stable generalized line bundles whose local index sequence
equals $\underline{b}$.  
\end{definition}

\begin{lemma} \label{Lemma:  FixedTypeLocus} 
	If $d - b_1 - \dots - b_k$ is odd, then $Z_{\underline{b}}$ is empty.  Otherwise, it is a locally closed, irreducible subset of dimension $$\operatorname{dim}(Z_{\underline{b}}) = g - (b_1-1) - \dots - (b_k-1).$$
\end{lemma}
\begin{proof}
	Set $b := b_1+\cdots + b_k$, the index of a generalized line bundle with local index sequence $\underline{b}$.  When $d -b$ is odd, the claim is just a restatement of 
Fact~\ref{Fact: TypeRestrict}.  Thus, assume $d-b$ is even.  We prove the lemma by parameterizing $Z_{\underline{b}}$ by an irreducible variety of the 
appropriate dimension.

Let $U \subset \operatorname{Hilb}^{b}(X_{\text{red}})$ denote the subset parameterizing closed subschemes $\Sigma$ supported at $k$ distinct points with multiplicities $b_1, \dots, b_k$. This subset is locally closed and irreducible of dimension $k$, because it is isomorphic to an open subset of the $k$-th symmetric power of $X_{\text{red}}$.
The hypothesis on $b$ implies that a generalized line bundle of index $b$ is stable, so we can define a morphism $U \times \operatorname{Pic}^{d-b}(X) \to \operatorname{M}(\calO_{X}, P_d)$ by the rule $(\Sigma, \calL) \mapsto \calI_{\Sigma} \otimes \calL$.  Here $\calI_{\Sigma}$ is the ideal sheaf of $\Sigma$.

Observe that $U$ has been chosen so that if $p_1, \dots, p_k$ is a collection of $k$ distinct points, then there is a unique closed subscheme $\Sigma$ corresponding to a
point of $U$ with the property that $b_{p_i}(\calI_{\Sigma}) = b_i$.  By Lemma~\ref{Lemma:  ActionIsTrans}, we may deduce that the image of 
$U \times \operatorname{Pic}^{d-b}(X) \to \operatorname{M}(\calO_{X}, P_d)$ is $Z_{\underline{b}}$. Note that the fiber over a point is an irreducible variety of dimension $b$, given by
$\ker(f^{*}: \operatorname{Pic}(X)\to \operatorname{Pic}(X'))$ where $f: X'\to X$ is the associated blow-up.  We can immediately conclude that $Z_{\underline{b}}$ is irreducible and constructible.

Furthermore, using the formula $b(\calI) = \operatorname{len}( \ShEnd(\calI)/\calO_{X})$ for index, it is easy to see that the index of a generalized line bundle 
can only increase under specialization.  Because $Z_{\underline{b}}$ is constructible, it follows that this subset is locally closed.  To complete the proof, we 
compute the dimension of $Z_{\underline{b}}$:
\begin{align*}
	\dim(Z_{\underline{b}}) &= \dim( U \times \operatorname{Pic}_{X}^{d-b}) - \dim( \text{Fiber} )\\
				&= (g+k)-b.
\end{align*}
\end{proof}

\begin{lemma} \label{Lemma: HowGLBSpecialize I}
	Let $X$ be a ribbon and $\calI$ a generalized line bundle.  Say that $p_0 \in X$ is a point such that $b_{p_0}(\calI) = b_0+2$ for some $b_0 \ge 0$.  Then there 
is a $\Spec(k[[\alpha]])$-flat, finitely presented $\calO_{X \times \Spec(k[[\alpha]])}$-module $\calI_\alpha$ whose special fiber is isomorphic to $\calI$
and whose generic fiber $\calI_{\alpha}[\alpha^{-1}]$ is a generalized line bundle with local index
\begin{displaymath}
	b_{p}(\calI_{\alpha}[\alpha^{-1}]) = \begin{cases}
					b_{p}(\calI) & \text{if $p \ne p_0$;}\\
					b_{0} & \text{if $p=p_0$.}
					\end{cases}
\end{displaymath}
\end{lemma}
\begin{proof}
	Given $\calI$, we first exhibit another generalized line bundle $\calI'$ with the same local index as $\calI$ at every point $p \in X$ that admits a suitable deformation.  Then we use
Lemma~\ref{Lemma:  ActionIsTrans} to deform $\calI$.

Say $p_1, \dots, p_k$ are the points $p$ with $b_{p}(\calI) \ne 0$, labeled so $p_1 = p_0$.  Define $Z$ to be the (unique) closed subscheme of $X_{\text{red}}$ that is supported at $p_1, \dots, p_k$ and has length $b_{p_i}(\calI)$ at $p_i$.  If we consider $Z$ as a closed subscheme of $X$, then the ideal sheaf $\calI' := \calI_Z$ is a generalized line bundle with the property that $b_{p}(\calI) = b_{p}(\calI')$ for all $p \in X$.  We deform $\calI'$ by deforming $Z$ as a closed subscheme.  

Fix an open affine neighborhood $U$ of $p_1$ with the property that $p_i \notin U$ for $i \neq 1$ and $Z|_{U}$ (considered as a subscheme of $X$) is defined by $(\epsilon, s^{b_0+2})$
for regular functions $\epsilon, s$ with $\epsilon$ in the nilradical.  A deformation of $(\epsilon, s^{b_0+2})$ over $k[[\alpha]]$ is given by the ideal generated by
\begin{gather} \label{Eqn: Deformation}
	 \epsilon (s-\alpha),\\
	s^{b_0} (\epsilon - \alpha^{b_0+1} (s-\alpha) ),  \notag \\
	s^{b_0} (s-\alpha)^2,  \notag \\
	\epsilon - \alpha s^{b_0+1} + \alpha^2 s^{b_0}. \notag
\end{gather}
One may verify that the restriction of this ideal to the generic fiber equals $(\epsilon, s^{b_0}) \cap (\epsilon - \alpha^{b_0+1} (s-\alpha), (s-\alpha)^2)$,
which geometrically is the union of a degree $2$ Cartier divisor supported at $\{ s = \alpha\}$ and the degree $b_0$ closed subscheme contained in the reduced subcurve and supported at $p_1$.
In particular, the degree of the generic fiber equals the degree of the special fiber, so this family of closed subschemes is flat.  We can extend this deformation of $Z_{U}$ to a deformation $Z_{\alpha}$ of $Z$ by defining $Z_{\alpha}$ to be the constant deformation
away from $U$.  The associated ideal $\calI'_{\alpha} := \calI_{Z_{\alpha}}$ is a deformation of $\calI'$ satisfying the conditions of the lemma.

This proves the lemma when $\calI = \calI'$.  To deduce the general case, by Lemma~\ref{Lemma:  ActionIsTrans} there exists a line bundle $\calL$ such that
$\calI = \calI' \otimes \calL$.  If we define $\calI_{\alpha}$ to be the tensor product of $\calI'_{\alpha}$ with the constant family with fiber $\calL$, then $\calI_{\alpha}$ 
satisfies the conditions of the lemma.  This completes the proof.
\end{proof}

\begin{theorem} \label{Thm: LineBundleComp}
	Let $X$ be a ribbon and $i$ an integer satisfying
\begin{displaymath}
	1 \le i \le \begin{cases}
				(g+2)/2 - \bar{g} & \text{if $d$ is even;}\\
				(g+1)/2 - \bar{g} & \text{if $d$ is odd.}
			\end{cases}
\end{displaymath}
	Define $\bar{Z}_{i} \subset \operatorname{M}(\calO_{X}, P_d)$ to be the Zariski closure of 
$Z_{\underline{b}}$, where $\underline{b} = (1, \dots, 1)$ is the sequence of $1$'s with length equal to 
$2 i -2$ if $d$ is even and $2 i -1$ if $d$ is odd. 

	Then $\bar{Z}_{i}$ is a $g$-dimensional irreducible component of $\operatorname{M}(\calO_{X}, P_d)$.  Furthermore,
if $\bar{Z} \subset \operatorname{M}(\calO_{X}, P_d)$ is any irreducible component that
contains a stable generalized line bundle, then $\bar{Z} = \bar{Z}_i$ for some $i$.
\end{theorem}
\begin{proof}
	The theorem follows from Lemmas \ref{Lemma:  GLBIsOpen}, \ref{Lemma:  FixedTypeLocus}, and \ref{Lemma: HowGLBSpecialize I}.
Observe that a repeated application of Lemma~\ref{Lemma: HowGLBSpecialize I} shows that $\cup \bar{Z}_i$ contains the locus of generalized line bundles.  Furthermore, 
the subsets $\bar{Z}_i$ are all irreducible and of dimension $g$ by Lemma~\ref{Lemma: FixedTypeLocus}.  We now use these facts to prove the proposition.

First, let us prove that every irreducible component that contains a stable generalized line bundle is of the form $\bar{Z}_i$ for some $i$.  Say $\bar{Z}$ is
an irreducible component that contains a stable generalized line bundle. The subset of $\bar{Z}$ consisting of stable generalized line bundles
is non-empty and open (Lemma~\ref{Lemma:  GLBIsOpen}), hence dense.  We may conclude that $\bar{Z}$ is contained in the union 
$\cup \bar{Z}_j$, but each of the $\bar{Z}_j$'s is irreducible, so this is only possible if $\bar{Z} = \bar{Z}_i$
for some $i$.

This shows that some of the $\bar{Z}_i$'s are irreducible components of $\operatorname{M}(\calO_{X}, P_d)$.  To complete
the proof, we must show that every $\bar{Z}_i$ is a component, so let $\bar{Z}_i$ be given.  Certainly this subset
is contained in some component, which we have shown must be of the form $\bar{Z}_j$ for some $j$. In other words,  
we have $\bar{Z}_i \subset \bar{Z}_j$. As both  are irreducible of dimension $g$, this is only possible
if $\bar{Z}_i = \bar{Z}_j$, showing that $\bar{Z}_i$ is a component.  This completes the proof.
\end{proof}

Theorem~\ref{Thm: LineBundleComp} describes all the irreducible components of $\operatorname{M}(\calO_{X}, P_d)$ 
that contain a stable generalized line bundle. What about the locus parameterizing sheaves of the form
$i_* \calE$ for $\calE$ a semi-stable rank $2$ vector bundle on $X_{\text{red}}$?  We prove the following result.
\begin{theorem} \label{Thm: VectorComponent}
	Let $X$ be a ribbon.  If $g \le 2 \bar{g}-1$, then no component of $\operatorname{M}(\calO_{X}, P_d)$
	contains a stable generalized line bundle.  Furthermore, the moduli space 
	is empty when $\bar{g}=0$ and $d-g$ is even. Otherwise $\operatorname{M}(\calO_{X}, P_d)$ is irreducible of dimension
\begin{equation} \label{Eqn: VectBunDim}
	\dim \operatorname{M}(\calO_{X}, P_d) = \begin{cases}
								0 & \text{if $\bar{g}=0$;} \\
								1 & \text{if $\bar{g}=1$, $d-g$ is even;}\\
								2 & \text{if $\bar{g}=1$, $d-g$ is odd;}\\
								4 \bar{g} - 3 & \text{if $\bar{g} \ge 2$.}
							\end{cases}
\end{equation}

	If $g > 2 \bar{g}-1$, then $\operatorname{M}(\calO_{X}, P_d)$ has at most one irreducible component that
does not contain a stable generalized line bundle. When this component exists, it has dimension $4 \bar{g} - 3$.
This component does not exist when $\bar{g} \le 1$, but does exist when $\bar{g} \ge 2$ and $4\bar{g}-3\ge g$.
\end{theorem}
Theorem~\ref{Thm: VectorComponent} follows from known results about moduli of vector bundles on non-singular curves, so before giving the proof, we recall the 
relevant facts from \cite{potier}.

Computing Euler characteristics, we see that the direct image of a rank $2$ semi-stable vector bundle $\calE$ on $X_{\text{red}}$ corresponds to a point of 
$\operatorname{M}(\calO_{X}, P_d)$ precisely when $\deg(\calE) = d + 2 \bar{g} - g - 1$.  Thus, if we set
\begin{displaymath}
	e := d + 2 \bar{g} - g - 1
\end{displaymath}
and $\operatorname{M}(2,e)$ equal to the coarse moduli space of semi-stable rank $2$ vector bundles on $X_{\text{red}}$ of degree $e$
(which exists by, say, Simpson's work), then the rule $\calE \mapsto i_* \calE$ defines a closed embedding
$\operatorname{M}(2,e) \to \operatorname{M}(\calO_{X}, P_d)$ with image equal to the subset 
parameterizing sheaves of the form $i_* \calE$.  The geometry of $\operatorname{M}(2,e)$ 
depends on the genus of $X_{\text{red}}$.  There are three cases to consider: $\bar{g}=0$,
$\bar{g}=1$ and $\bar{g} \ge 2$.  

When $\bar{g}=0$, every
rank $2$ vector bundle splits as a direct sum of line bundles, so there are no stable vector bundles and
the strictly semi-stable vector bundles are of the form $\calO_{\bbP^1}(e/2) \oplus \calO_{\bbP^1}(e/2)$
for $e$ even.  Thus, $\operatorname{M}(2,e)$ is empty when $e$ is odd (i.e. $d-g$ is even) and equal to a point
when $e$ is even (i.e. $d-g$ is odd).

The next case to consider is  $\bar{g}=1$.  Again, the geometry of $\operatorname{M}(2,e)$
depends on the parity of $e$.  For $e$ odd (i.e. $d-g$ even), every semi-stable rank $2$
vector bundle of degree $e$ is stable and the determinant map defines an isomorphism
$\operatorname{M}(2,e) \to \operatorname{Pic}^{e}(X_{\text{red}})$ (\cite[Thm.~8.29]{potier}).  In particular, 
$\operatorname{M}(2, e)$ is irreducible of dimension $1$.  By contrast, there are
no stable vector bundles of odd degree on $X_{\text{red}}$ (\cite[pg.~160-161]{potier}), and $\operatorname{M}(2,e)$ is irreducible  of dimension $2$ (\cite[Thm.~8.29]{potier}).

The final case is when $\bar{g} \ge 2$.  In this case, there always exists a stable
rank $2$ vector bundle of degree $e$ (\cite[Thm.~8.28]{potier}), and $\operatorname{M}(2,e)$ is irreducible of dimension
$4 \bar{g}-3$ (\cite[Thm.~8.14, 8.26]{potier}).  We now use these facts to prove the theorem.  

\begin{proof}[Proof of Theorem~\ref{Thm: VectorComponent}]
	The case where $g \le 2 \bar{g}-1$ follows from Theorem~\ref{Thm: StabilityThm}.  That theorem asserts that there are no stable generalized line bundles, so the natural map $\operatorname{M}(2, e) \to \operatorname{M}(\calO_{X}, P_d)$ is a set-theoretic bijection, and the claim follows from 
the results about $\operatorname{M}(2,e)$ just reviewed.  For the remainder of the proof, we assume $g > 2 \bar{g} -1$.  

	The case $\bar{g}=0$ can be treated without difficulty.  By the classification of vector bundles on $\bbP^1$,
it is enough to show that if $d-g$ is odd, then the point of $\operatorname{M}(\calO_{X}, P_d)$ corresponding to the
direct image of 
\begin{displaymath}
	\calO(e/2) \oplus \calO(e/2)
\end{displaymath}
is contained in the closure of the locus of generalized line bundles.  The vector bundle in question is $\operatorname{Gr}$-equivalent to
a generalized line bundle of index $g+1$, and every such generalized line bundle is the specialization
of a stable generalized line bundle by Lemma~\ref{Lemma: HowGLBSpecialize I}, combined with 
Corollary~\ref{Cor: RatlStability}.

	We now turn our attention to the case where $\bar{g}=1$.  When $d-g$ is odd, the claim follows from essentially the same
argument as in the $\bar{g}=0$: there are no stable vector bundles of degree $e$ on $X_{\text{red}}$ and every strictly semi-stable rank $2$ vector
bundle is $\operatorname{Gr}$-equivalent to a generalized line bundle of index $g-1$.  However, when $d-g$ is even, a different argument is needed.  

	Consider first the special case where we not only assume $\bar{g}=1$ and $d-g$ is even, but also $g=2$. Then there are exactly two types of semi-stable
sheaves with degree $d$: line bundles on $X$ and sheaves of the form $i_* \calE$ for $\calE$ a rank $2$ stable
vector bundle on $X_{\text{red}}$.  The moduli space $\operatorname{Pic}^{d}(X)$
is not projective, hence its image in $\operatorname{M}(\calO_{X}, P_d)$ is not closed (because $\operatorname{M}(\calO_{X}, P_d)$
is projective).  Therefore, its closure in $\operatorname{M}(\calO_{X}, P_d)$ must contain some stable vector bundle.  But the natural action of $\operatorname{Pic}^{0}(X)$
on the set of stable rank $2$ vector bundles of degree $e$ is transitive, so if the closure of $\operatorname{Pic}^{d}(X)$ contains one such vector bundle, then it 
contains all such vector bundles.  This completes the proof when $g=2$.

	The case where $g>2$ can be deduced from this.  Suppose now that we relax the condition $g=2$ to $g \ge 2$. Blowing-up a length $g-2$ 
closed subscheme contained in $X_{\text{red}}$, we obtain a finite morphism $f \colon X' \to X$ with $g(X')=2$.  Now suppose that
we are given a sheaf of the form $i_* \calE$ for $\calE$ a stable rank $2$ vector bundle on $X_{\text{red}}$.  We need to show that
$i_* \calE$ is the specialization of a stable generalized line bundle.  By what we have just shown, there exists a family of sheaves
$\calI'$ on $X' \times \Spec(k[[\alpha]])$ whose generic fiber is a stable line bundle and whose special fiber is isomorphic to 
$i'_{*} \calE$.  The direct image of this family under $f$ realizes $i_* \calE$ as the specialization of a stable generalized line bundle.  
This completes the proof in the case that $\bar{g}=1$.

The remaining case is where $\bar{g} \ge 2$ and $4 \bar{g} - 3\ge g$. Under these conditions, we wish to 
show that the closure of the image of $\operatorname{M}(2,e) \to \operatorname{M}(\calO_{X}, P_d)$ is an irreducible component.  Because
$\operatorname{M}(2, e)$ is itself irreducible, it is enough to show that the image is not contained in 
$\cup \bar{Z}_j$.  This follows from a dimension count: $\cup \bar{Z}_j$ has dimension equal to $g$, which 
is smaller than or equal to the dimension of $\operatorname{M}(2, e)$ by assumption.
\end{proof}

Observe that the proposition shows that there are ribbons with the property that every semi-stable rank $2$ vector bundle is the specialization of a stable generalized line bundle, and there are ribbons that do not have this property.  However, the proposition does exhaustively analyze this phenomenon: the proposition says nothing when the inequalities 
$\bar{g} \ge 2$ and $g > 4\bar{g}-3$ are both satisfied.  Thus, it would be interesting to know the answer to the following question:

\begin{question}
	Let $X$ be a ribbon such that $\bar{g} \ge 2$ and $g > 4 \bar{g} - 3$.  Does there exist an irreducible component of
$\operatorname{M}(\calO_{X}, P_d)$ whose general element is a rank $2$ vector bundle on $X_{\red}$?
\end{question}

We now prove that $\operatorname{M}(\calO_{X}, P_d)$ is connected.  To establish this, we need another
lemma about deformations of generalized line bundles.
\begin{lemma} \label{Lemma: HowGLBSpecialize II}
	Let $X$ be a ribbon and $\calI$ a generalized line bundle. If the local index sequence
	of $\calI$ is $(b_1+1, b_2, \dots, b_k)$, then $\calI$ is the specialization of a generalized line bundle
	with local index sequence $(1, b_1, b_2, \dots, b_k)$.  
\end{lemma}
\begin{proof}
	The proof is similar to the proof of Lemma~\ref{Lemma: HowGLBSpecialize I}.  The essential point is to show
	that the ideal $(\epsilon, s^{b_1+1})$ can be deformed to an ideal with local index $b_1$ at $\{ \epsilon = s = 0\}$ and
	local index $1$ at a second point.  One such deformation over $k[[\alpha]]$ is given by the ideal generated by
	\begin{gather}
		\epsilon, \\
		s^{b_0}(s-\alpha). \notag
	\end{gather}
\end{proof}

\begin{theorem}\label{theorem: Conn}
	For a ribbon $X$, the  moduli space $\operatorname{M}(\calO_{X}, P_d)$ is connected.
\end{theorem}
\begin{proof}
The case where $g - 2\bar{g} + 1 \leq 0$ can be dispensed with immediately.  When this inequality holds, there are no stable generalized line bundles (Thm.~\ref{Thm: StabilityThm}), and $\operatorname{M}(\calO_{X}, P_d)$
is, in fact, irreducible.  Thus, for the remainder of the proof we assume $g - 2\bar{g} + 1 > 0$. 

In this case, we prove connectivity by exhibiting a point lying in every irreducible component.  Set $e := d-g+2 \bar{g}-1$.  There are two separate cases to consider: $e$ even and $e$ odd.  When $e$ is even, there exist strictly semi-stable
generalized line bundles with local index sequence $(d-e)$. Fix one such sheaf $\calI$.   We claim that the corresponding
 point $x_0$ of $\operatorname{M}(\calO_{X}, P_d)$  lies in every irreducible component of $\operatorname{M}(\calO_{X}, P_d)$.  Using Theorem~\ref{Thm: LineBundleComp},  repeated applications of Lemmas \ref{Lemma: HowGLBSpecialize I} and \ref{Lemma: HowGLBSpecialize II} show that $x_0$ lies in every irreducible component containing a 
stable generalized line bundle.  There is at most one other irreducible component, which, when it exists, parameterizes
rank $2$ vector bundles on $X_{\text{red}}$.  In any case, $x_0$ must lie in the locus of rank $2$ vector bundles on $X_{\text{red}}$ because $\calI$ is $\operatorname{Gr}$-equivalent to such a sheaf (Lemma~\ref{Lemma: FilterGLB}). 
This proves the proposition when $e$ is even.

Now suppose that $e$ is odd.  If we replace the strictly semi-stable generalized line bundle with local index  sequence $(d-e)$ 
with a stable generalized line bundle of local index $(d-e-1)$, then a simple modification of the argument given in the
$e$ even case shows that there is a point $x_0$ lying in every irreducible component that contains
a stable generalized line bundle.  To complete the proof, we must show that the locus of stable rank $2$ vector bundles
in $\operatorname{M}(X, P_d)$ has non-empty intersection with the locus of stable generalized line bundle.  Because
there are no semi-stable rank $2$ vector bundles on $\bbP^1$, we may assume $\bar{g} \ge 1$.

Now suppose first that $g - 2 \bar{g} + 1 =1$.  Then there are at most two types of sheaves corresponding to points of 
$\operatorname{M}(\calO_{X}, P_d)$: stable line bundles of degree $d$ and stable rank $2$ vector bundles on $X_{\text{red}}$.
The locus of line bundles is not closed in $\operatorname{M}(\calO_{X}, P_d)$ because the Simpson moduli space
is projective, but the line bundle locus is not proper.  Thus, the closure of the locus of line bundles must 
contain at least one point corresponding to a rank $2$ vector bundle on $X_{\text{red}}$, which is what we 
wished to show.

When $g - 2 \bar{g} + 1>1$, we can reduce to the previous case.  Indeed, a suitable blow-up $f \colon X' \to X$
of $X$ has the property that $g(X') - 2 \bar{g}(X')+1=1$.  We have just shown that there is a family of stable line bundles on 
$X'$ specializing to a stable rank $2$ bundle on $X'_{\text{red}} = X_{\text{red}}$, and the direct image of this family under $f$ is
a family of stable generalized line bundles specializing to a rank $2$ vector bundle on $X_{\text{red}}$.  In other words,
the intersection of the locus of generalized line bundles in $\operatorname{M}(\calO_{X}, P_d)$ has non-trivial intersection
with the locus of vector bundles on $X_{\text{red}}$, completing the proof.
\end{proof}

\subsection{Local Structure} \label{Subsec: LocStr}
We now turn our attention to the local structure of the Simpson moduli space. We compute
the tangent space dimension of $\operatorname{M}(\calO_{X}, P_d)$ at a point corresponding to 
a stable sheaf and then apply this result to determine the smooth locus of 
the moduli space.    The specific tangent space computation we give is the following proposition:

\begin{proposition} \label{Prop: MainTanSpace}
	Let $x$ be a point of $\operatorname{M}_{\s}(\calO_{X}, P_d)$.  If $x$ corresponds
to a stable generalized line bundle $\calI$, then we have
\begin{displaymath}
	\dim \operatorname{T}_{x} \operatorname{M}(\calO_{X}, P_d) = g + b(\calI).
\end{displaymath}
	If $x$ corresponds to the direct image $i_* \calE$ of a stable rank $2$ vector bundle $\calE$ on
$X_{\red}$, then
	\begin{align*}
		\dim \operatorname{T}_{x} \operatorname{M}( \calO_{X}, P_d) &= 4 \bar{g} - 3 +  h^{0}(X_{\red}, \ShEnd(\calE) \otimes \calN^{-1}) \\
											&= 4 g + 5 - 8 \bar{g}  \text{ if $g \ge 4 \bar{g} - 2.$}
	\end{align*}
\end{proposition}
The proof of the proposition is broken up into two separate lemmas: one computing the dimension when $x$ is
a generalized line bundle (Lemma~\ref{Lemma: TanAtGLB}) and one when $x$ corresponds to a rank $2$ vector
bundle on $X_{\text{red}}$ (Lemma~\ref{Lemma:  TanAtVec}).  In both cases, the starting point is the
identification of the tangent space with the cohomology group $\Ext^{1}(\calF, \calF)$.  Let us begin by 
recalling how this identification works.

Elements of the tangent space are in natural bijection with morphisms $\Spec(k[\alpha]/(\alpha^2)) \to \operatorname{M}(\calO_{X}, P_d)$ sending the closed point to $x$.  Because $x$ lies in the stable locus, these morphisms in turn are in natural bijection with first-order deformations of $\calF$ (i.e. deformations over
$k[\alpha]/(\alpha^2)$).  The isomophism $\operatorname{T}_{x} \operatorname{M}(\calO_{X}, P_d) \cong \Ext^{1}(\calF, \calF)$ is constructed by exhibiting a bijection between first-order deformations of $\calF$ and elements of $\Ext^{1}(\calF, \calF)$.  

If $\calF_1$ is a deformation of $\calF$, then tensoring the short exact sequence
$k \cong (\alpha) \hookrightarrow k[\alpha]/(\alpha^2) \twoheadrightarrow k$ with $\calF_1$ yields the sequence 
\begin{equation}
	\calF \cong \alpha \cdot \calF_1 \hookrightarrow \calF_1 \twoheadrightarrow \calF_1/ \alpha \cdot \calF_1 \cong \calF.
\end{equation}
This is an extension of $\calF$ by $\calF$, and hence defines an element $c(\calF_1)$ of $\Ext^{1}(\calF, \calF)$. One shows that the rule $\calF_1 \mapsto c(\calF_1)$ 
is well-defined and the induced map $$\operatorname{T}_{x} \operatorname{M}(\calO_{X}, P_d) \to \Ext^{1}(\calF, \calF)$$ is a bijection.  As an aside, we remark that when $x$ is strictly semi-stable, there is no longer a canonical identification $\operatorname{T}_{x} \operatorname{M}( \calO_{X}, P_d) \cong \Ext^{1}(\calF, \calF)$.
Rather, there is a more complicated description of the tangent space involving a natural action of $\Aut(\calF)$ on $\Ext^{1}(\calF, \calF)$.  

We now compute $\Ext^{1}(\calI, \calI)$ for $\calI$ a generalized line bundle.
\begin{lemma} \label{Lemma:  TanAtGLB}
	If $\calI$ is a generalized line bundle, then we have 
	\begin{equation} \label{Eqn: TanAtGLB I}		
		\operatorname{dim}( \Ext^{1}(\calI, \calI) ) = g + b(\calI) + h^{0}(X', \calO_{X'}) - 1,
	\end{equation}
	where $X'$ is the associated blow-up of $\calI$.  
	
	If we additionally assume $2 \bar{g} \le g$, then this formula simplifies to 
	\begin{equation} \label{Eqn: TanAtGLB II}
		\operatorname{dim}( \Ext^{1}(\calI, \calI) ) = g + b(\calI).
	\end{equation}		
\end{lemma}
\begin{proof}
	An inspection of the local-to-global spectral sequence 
$$H^{p}( X, \ShExt^{q}(\calI, \calI)) \Rightarrow \Ext^{p+q}(\calF, \calF)$$ computing $\Ext$ shows that there is a 
short exact sequence
\begin{displaymath}
	H^{1}(X, \ShEnd(\calI)) \hookrightarrow \Ext^{1}(\calI, \calI) \twoheadrightarrow  H^{0}(X, \ShExt^{1}(\calI, \calI)).
\end{displaymath}
We prove the proposition by computing the right-most term and the left-most term in the above sequence. If we write 
$\calI = f_{*}(\calI')$ for a line bundle $\calI'$ on a blow-up $f \colon X' \to X$, then the canonical map 
$f_{*}(\calO_{X'}) \to \ShEnd(\calI)$ is an isomorphism.  As a consequence, 
$H^{1}(X, \ShEnd(\calI)) = H^{1}(X', \calO_{X'})$, which is of dimension $g-b(\calI) + h^{0}(X', \calO_{X'})-1$.  

We also need to compute $H^{0}(X, \ShExt^{1}( \calI, \calI))$.  The sheaf $\ShExt^{1}( \calI, \calI)$ is supported on the 
points $p$ satisfying $b_{p}(\calI)>0$.   If we label these points as $p_1, \dots, p_n$, then the space of global sections decomposes as
\begin{displaymath}
	H^{0}(X, \ShExt^{1}(\calI, \calI)) = \oplus_{j=1}^{n} \Ext^{1}(\widehat{\calI}_{p_j},\widehat{\calI}_{p_j}),
\end{displaymath}
where $\widehat{\calI}_{p_j}$ is the restriction of $\calI$ to the completed local ring $\widehat{\calO}_{X, p_j}$. 
Using the free resolution \eqref{Eqn: FreeRes}, one computes $\dim \Ext^{1}(\widehat{\calI}_{p_j},\widehat{\calI}_{p_j}) = 
2 b_{p_j}(\calI)$.  Summing over all $j$, we have
\begin{align*}
	\dim \Ext^{1}(\calI, \calI) &= (g - b(\calI) + h^{0}(X', \calO_{X'})-1)+ (2 b_{p_1}(\calI) + \dots + 2 b_{p_n}(\calI)) \\
					&= g + b(\calI) + h^{0}(X', \calO_{X'})-1.
\end{align*}
	This establishes \eqref{Eqn: TanAtGLB I}.  When $2 \bar{g} \le g$, the nilradical has negative degree, and hence no non-zero global sections.  In particular, $h^{0}(X', \calO_{X'}) -1=0$, proving that \eqref{Eqn: TanAtGLB II} holds.
\end{proof}

We now turn our attention to sheaves of the form $\calF = i_* \calE$.
\begin{lemma} \label{Lemma:  TanAtVec}
	Let $\calE$ be a stable rank $2$ vector bundle on $X_{\red}$.  Then we have
	\begin{equation}
		\dim \Ext^{1}( i_*\calE, i_*\calE) = 4 \bar{g} - 3 + h^{0}(X_{\red}, \calN^{-1} \otimes \ShEnd(\calE)).
	\end{equation}
	If we additionally assume $g \ge 4 \bar{g} - 2$, then this formula simplifies to
	\begin{equation} \label{Eqn: BigGDim}
		\dim \Ext^{1}( i_*\calE, i_* \calE) = 4 g + 5 - 8 \bar{g}.
	\end{equation}

\end{lemma}
\begin{proof}
	Like the previous lemma, this is proven using a spectral sequence argument.   For any two
	$\calO_{X_{\text{red}}}$-modules $\calA$ and $\calB$, adjunction provides a canonical 
	identification $\Hom_{\calO_{X}}(i_* \calA, i_* \calB) = \Hom_{\calO_{X_{\text{red}}}}( i^{*} i_{*} \calA, \calB)$,
	hence the groups $\Ext^{n}( i_* \calA, i_* \calB) $ and $\Ext^{n}( i^{*} i_{*} \calA,\calB)$ are isomorphic. We
	compute $\Ext^{1}( i_* \calE, i_* \calE)$ by working with a spectral sequence describing $\Ext^{1}( i^{*}i_{*} \calE, \calE)$.

	The functor $\Hom( i^{*} i_{*} \underline{\phantom{\calE}}, \calE)$ is the composition of the functors
	$G := i^{*} i_{*}$ and $F := \Hom_{\calO_{X_{\text{red}}}}(\underline{\phantom{\calE}}, \calE)$, so there  is
	a Grothendieck spectral sequence:
	\begin{displaymath}
		\Ext^{p}( \Tor_{q}( i_{*}(\calO_{X_{\text{red}}}), i_* \calE), \calE) \Rightarrow \Ext^{p+q}( i^{*}i_{*} \calE, \calE). 
	\end{displaymath}
	The first four terms of the associated exact sequence of low degree terms are:
	\begin{equation} \label{Eqn: LowTermSeq}
		\Ext^{1}(\calE, \calE) \hookrightarrow \Ext^{1}( i^{*} i_*\calE, \calE) \to \Hom( \Tor_{1}( i_* \calO_{X_{\text{red}}}), \calE) \to
		\Ext^{2}(\calE, \calE).
	\end{equation}
	Because $X_{\text{red}}$ is a non-singular curve, the last term $\Ext^{2}(\calE, \calE)$ vanishes.  We can also compute the
	second-to-last term.  Associated to the short exact sequence $\calN \hookrightarrow \calO_{X} \twoheadrightarrow \calO_{X_{\text{red}}}$
	is a long exact sequence of $\operatorname{Tor}(\underline{\phantom{\calE}}, i_*\calE)$-groups whose connecting 	
	map 
	\begin{displaymath}
		\partial \colon \Hom( \calN \otimes \calE, \calE) \to \Hom( \Tor_{1}( i_* \calO_{X_{\text{red}}}, i_{*}\calE))
	\end{displaymath}
	is an isomorphism.  We now compute the dimension of $\Ext^{1}( i_* \calE, i_* \calE)$ using Sequence~\eqref{Eqn: LowTermSeq}:
	\begin{align*}
		\dim \Ext^{1}( i_* \calE, i_* \calE) 	&= \dim \Ext^{1}( i^{*} i_*\calE, \calE) \\
							&= \dim \Ext^{1}(\calE, \calE) + \dim \Hom( \calN \otimes \calE, \calE) \\
							&= 4\bar{g} - 3 + h^{0}(X_{\text{red}},  \calN^{-1} \otimes \ShEnd(\calE)).
	\end{align*} 
	The group $\Ext^{1}(\calE, \calE)$ is computed by the Riemann--Roch formula. (Note: $\Hom(\calE, \calE)$ is 1-dimensional as $\calE$ is stable.)  This proves the first part of the proposition.

	Now assume $g \ge 4 \bar{g} - 2$.  Because $\calE$ is stable, the rank $4$ vector bundle $\ShEnd(\calE) \otimes \calN^{-1}$ is semi-stable (\cite[Cor.~6.4.14]{rlaz}) of  degree  $4+4g-8\bar{g}$. Using Serre duality on $X_{\red}$, one checks  
	that the higher cohomology of this bundle vanishes.  Formula~\eqref{Eqn: BigGDim}
	now follows from the Riemann-Roch formula, completing the proof.
\end{proof}

 One immediate corollary of the proposition is the following:
\begin{corollary} \label{Cor: SmoothLocus}
	Let $X$ be a ribbon.  If $\bar{g} \ge 2$ and $g \ge 4 \bar{g} - 2$, then the smooth locus of $\operatorname{M}(\calO_{X}, P_d)$ is equal to the open subset parameterizing line bundles on $X$.
\end{corollary}
\begin{proof}
	First, let us prove the weaker claim concerning the smooth locus of the space $\operatorname{M}_{\text{s}}(\calO_{X}, P_d)$ of stable sheaves.  It is enough to show that if $x \in \operatorname{M}_{\text{s}}(\calO_{X}, P_d)$, then the tangent space dimension of $\operatorname{M}(\calO_{X}, P_d)$ equals the local (topological) dimension if and only if $x$ 
corresponds to a line bundle.  Theorems \ref{Thm: LineBundleComp},  \ref{Thm: VectorComponent} together with 
Proposition~\ref{Prop: MainTanSpace} show that equality must fail except possibly in the following cases: when $x$ corresponds to a stable line bundle and when $x$ corresponds to a stable  rank $2$ vector bundle on $X_{\text{red}}$.  We must show that the second case
cannot occur.

If $x$ corresponds to a stable rank $2$ vector bundle on $X_{\text{red}}$, then the tangent space dimension is: 
\begin{displaymath}
	\dim\operatorname{T}_{x} \operatorname{M}_{\text{s}}(\calO_{X}, P_d)  = 4 g + 5 - 8 \bar{g}.
\end{displaymath}
We have not computed the local dimension of $\operatorname{M}(\calO_{X}, P_d)$ at $x$, but 
Theorems~\ref{Thm: LineBundleComp} and \ref{Thm: VectorComponent} show this
 local dimension at $x$ is either $g$ or $4 \bar{g} - 3$. A direct computation shows that both
numbers are strictly smaller than $4 g + 5 - 8 \bar{g}$, proving the claim concerning the locus
of stable sheaves.

What about the strictly semi-stable locus?  Because $\bar{g} \ge 2$, every strictly semi-stable point is the specialization of 
a stable point that is not a line bundle.  Indeed, the strictly semi-stable locus is contained in the image of $\operatorname{M}(2, e)$, and the stable locus is dense in $\operatorname{M}(2,e)$.  Because the singular locus is closed, we can conclude that the strictly semi-stable locus is contained in the 
singular locus, completing the proof.
\end{proof}
Observe that the hypothesis that $g \geq 4 \bar{g}-2$ is used to compute the dimension of the tangent space of $\operatorname{M}(\calO_{X}, P_d)$ at a point
corresponding to a stable rank $2$ vector bundle on $X_{\text{red}}$.  The method of proof can be used to describe the singular locus of 
$\operatorname{M}(\calO_{X}, P_d)$ under weaker hypotheses, but then the conclusion becomes more difficult to state: if $g \le 4 \bar{g} -3$ is allowed, 
then the smooth locus may contain stable rank $2$ vector bundles.  For example, if $g$ is sufficiently negative and $\bar{g} \ge 2$, then there are no semi-stable generalized line bundles, so $\operatorname{M}(\calO_{X}, P_d)$ has
a unique component of dimension $4 \bar{g}-3$ whose smooth locus contains the locus of stable rank $2$ vector bundles on $X_{\text{red}}$.  The
proof also makes use of the fact that every strictly semi-stable sheaf is the specialization of a stable sheaf.  In general, this 
condition may fail to hold (e.g. $\bar{g}=1$ and $g=2$). When this occurs, a more delicate analysis of the semi-stable locus is needed.

\bibliographystyle{amsalpha}
\bibliography{bibl}
\end{document}